\documentclass[12pt]{article}
\input{amssym}
\newtheorem{prethm}{{\bf Theorem}}

\newenvironment{thm}{\begin{prethm}{\hspace{-0.5
em}{\bf.}}}{\end{prethm}}
\newtheorem{precor}{{\bf Corollary}}

\newenvironment{cor}{\begin{precor}{\hspace{-0.5
em}{\bf.}}}{\end{precor}}
\newtheorem{preprop}{{\bf Proposition}}

\newtheorem{preque}{{\bf Question}}

\newtheorem{preques}{{\bf Question}}

\newtheorem{prelemma}{{\bf Lemma}}

\newenvironment{lemma}{\begin{prelemma}{\hspace{-0.5
em}{\bf.}}}{\end{prelemma}}
\newtheorem{prefact}{{\bf Fact}}

\newtheorem{preobs}{{\bf Observation}}

\newenvironment{obs}{\begin{preobs}{\hspace{-0.5
em}{\bf.}}}{\end{preobs}}
\newtheorem{prefig}{{\bf Figure}}

\newtheorem{prelemm}{{\bf Lemma}}

\newtheorem{preex}{{\bf Example}}

\newtheorem{prepro}{{\bf Proposition}}

\newenvironment{pro}{\begin{prepro}{\hspace{-0.5
em}{\bf.}}}{\end{prepro}}
\newtheorem{prelem}{{\bf Theorem}}

\newenvironment{lem}{\begin{prelem}{\hspace{-0.5
em}{\bf.}}}{\end{prelem}}
\newtheorem{preproof}{{\bf Proof.}}

\newenvironment{proof}[1]{\begin{preproof}{\rm
               #1}\hfill{$\rule{2mm}{2mm}$}}{\end{preproof}}

\newtheorem{preconj}{{\bf Conjecture}}

\newtheorem{predeff}{{\bf Definition}}

\newenvironment{deff}{\begin{predeff}{\hspace{-0.5
em}{\bf.}}}{\end{predeff}}
\def\newpic#1{}
\date{}
\begin{document}

\title{
{\Large{\bf The Metric Dimension of Lexicographic Product of
Graphs}}}
%

{\small
\author{
{\sc Mohsen Jannesari} and {\sc Behnaz Omoomi  }\\
[1mm]
{\small \it  Department of Mathematical Sciences}\\
{\small \it  Isfahan University of Technology} \\
{\small \it 84156-83111, Isfahan, Iran}}




 \maketitle \baselineskip15truept

\begin{abstract}
For an ordered set $W=\{w_1,w_2,\ldots,w_k\}$ of vertices and a
vertex $v$ in a connected graph $G$, the ordered $k$-vector
$r(v|W):=(d(v,w_1),d(v,w_2),\ldots,d(v,w_k))$ is  called  the
(metric) representation of $v$ with respect to $W$, where $d(x,y)$
is the distance between the vertices $x$ and $y$. The set $W$ is
called  a resolving set for $G$ if distinct vertices of $G$ have
distinct representations with respect to $W$. The minimum
cardinality of a resolving set for $G$ is its metric dimension. In
this paper, we study the metric dimension of the lexicographic
product  of graphs $G$ and $H$, $G[H]$. First, we introduce a new
parameter which is called adjacency metric dimension of a graph.
Then, we obtain the metric dimension of $G[H]$ in terms of the
order of $G$ and the adjacency metric dimension of $H$.
\end{abstract}

{\bf Keywords:} Lexicographic product; Resolving set; Metric
dimension; Basis; Adjacency metric dimension.
\section{Introduction}
In this section, we present some definitions and known results
which are necessary  to prove our main theorems. Throughout this
paper, $G=(V,E)$ is a finite simple graph. We use $\overline G$
for the complement of  graph  $G$. The distance between two
vertices $u$ and $v$, denoted by $d_G(u,v)$, is the length of a
shortest path between $u$ and $v$ in $G$. Also, $N_G(v)$ is the
set of all neighbors of vertex $v$ in $G$. We write these simply
$d(u,v)$ and $N(v)$,  when no confusion can arise. The notations
$u\sim v$ and $u\nsim v$ denote the adjacency and none-adjacency
relation between $u$ and $v$, respectively. The symbols
$(v_1,v_2,\ldots, v_n)$ and $(v_1,v_2,\ldots,v_n,v_1)$ represent
a path of order $n$, $P_n$, and a cycle of order $n$, $C_n$,
respectively.


For an ordered set $W=\{w_1,w_2,\ldots,w_k\}\subseteq V(G)$ and a
vertex $v$ of $G$, the  $k$-vector
$$r(v|W):=(d(v,w_1),d(v,w_2),\ldots,d(v,w_k))$$
is called  the ({\it metric}) {\it representation}  of $v$ with
respect to $W$. The set $W$ is called a {\it resolving set} for
$G$ if distinct vertices have different representations. In this
case, we say  set $W$ resolves $G$. Elements in a resolving set
are called {\it landmarks}. A resolving set $W$ for $G$ with
minimum cardinality is  called a {\it basis} of $G$, and its
cardinality is the {\it metric dimension} of $G$, denoted by
$\beta(G)$.  The concept of (metric) representation is introduced
by Slater~\cite{Slater1975} (see~\cite{Harary}). For more results
related to these concepts
see~\cite{trees,bounds,sur1,Discrepancies}.
\par We say an ordered set $W$ {\it resolves} a set $T$ of vertices
in $G$, if the representations of vertices in $T$ are distinct
with respect to $W$. When $W=\{x\}$, we say that  vertex $x$
resolves $T$. To see that whether a given set $W$ is a resolving
set for $G$, it is sufficient to look at the representations of
vertices in $V(G)\backslash W$, because $w\in W$ is the unique
vertex of $G$ for which $d(w,w)=0$.
\par Two distinct vertices $u,v$ are said  {\it twins}
if $N(v)\backslash\{u\}=N(u)\backslash\{v\}$. It is called that
$u\equiv v$ if and only if $u=v$ or $u,v$ are twins.
In~\cite{extermal}, it is proved that ``$\equiv$" is an equivalent
relation. The equivalence class of vertex $v$ is denoted by
$v^*$. Hernando et al.~\cite{extermal} proved that $v^*$ is a
clique or an independent set in $G$. As in ~\cite{extermal}, we
say $v^*$ is of type (1), (K), or (N) if $v^*$ is a class of size
$1$, a clique of size at least $2$, or an independent set of size
at least $2$. We denote the number of equivalence classes of $G$
with respect to ``$\equiv$" by $\iota(G)$. We mean by
$\iota_{_K}(G)$ and $\iota_{_N}(G)$, the number of classes of
type (K) and type (N) in $G$, respectively. We also use $a(G)$ and
$b(G)$ for the number of all vertices in  $G$ which have at least
an adjacent twin and a none-adjacent twin vertex in $G$,
respectively.  On the other way, $a(G)$ is the number of all
vertices in the classes of type (K) and $b(G)$ is the number of
all vertices in the classes of type (N). Clearly,
$\iota(G)=n(G)-a(G)-b(G)+\iota_{_N}(G)+\iota_{_K}(G)$.
\begin{obs}~{\rm\cite{extermal}}\label{twins}
Suppose that $u,v$ are twins in a  graph $G$ and $W$ resolves $G$.
Then $u$ or $v$ is in $W$. Moreover, if $u\in W$ and $v\notin W$,
then $(W\setminus\{u\})\cup \{v\}$ also resolves $G$.
\end{obs}
\begin{lem}~\rm\cite{Ollerman}\label{B=1,B=n-1}
Let $G$ be a connected graph of order $n$.
Then,\begin{description}\item (i) $\beta(G)=1$ if and only if
$G=P_n$,\item (ii) $\beta(G)=n-1$ if and only if $G=K_n$.
\end{description}
\end{lem}
\begin{lem}~\rm\cite{K dimensional,idea}\label{B(P_n) B(C_N)}
\begin{description}
\item (i) If $n\notin\{3,6\}$, then $\beta(C_n\vee
K_1)=\lfloor{{2n+2}\over 5}\rfloor$, \item (ii)  If
$n\notin\{1,2,3,6\}$, then $\beta(P_n\vee K_1)=\lfloor{{2n+2}\over
5}\rfloor$.
\end{description}
\end{lem}
The metric dimension of cartesian product of graphs is studied by
Caseres et al. in ~\cite{cartesian product}. They  obtained the
metric dimension of cartesian product of graphs $G$ and $H$,
$G\square H$, where $G,H\in \{P_n,C_n,K_n\}$.


\par The {\it lexicographic product} of graphs $G$ and $H$,  denoted by
$G[H]$, is a graph with vertex set \linebreak $V(G)\times
V(H):=\{(v,u)~|~v\in V(G),u\in V(H)\}$, where  two vertices
$(v,u)$ and $(v',u')$ are  adjacent whenever,  $v\sim v'$, or
$v=v'$ and $u\sim u'$. When the order of $G$ is at least $2$, it
is easy to see that $G[H]$ is a connected graph if and only if
$G$ is a connected graph.

This paper is aimed to investigate the metric dimension of
lexicographic product of graphs.
The main goal of Section~\ref{Adjacency Resolving Sets} is
introducing  a new parameter, which  we call it adjacency metric
dimension. In Section~\ref{Lexicographic product},  we prove some
relations to determine the  metric dimension of lexicographic
product of graphs, $G[H]$, in terms of the order of $G$ and the
adjacency metric dimension of $H$. As a corollary of our main
theorems, we obtain the exact value of the metric dimension of
$G[H]$, where $G=C_n(n\geq 5)$ or $G=P_n (n\geq4)$, and
$H\in\{P_m,C_m,\overline P_m,\overline
C_m,K_{m_1,\ldots,m_t},\overline K_{m_1,\ldots,m_t}\}$.

\section{Adjacency Resolving Sets}\label{Adjacency Resolving Sets}

S. Khuller et al.~\cite{landmarks} have considered  the
application of the metric dimension of a connected graph in robot
navigation. In that sense, a robot moves from node to node of a
graph space. If the robot knows its distances to a sufficiently
large set of landmarks, its position on the graph is uniquely
determined. This suggest the problem of finding the fewest number
of landmarks needed, and where should be located, so that the
distances to the landmarks uniquely determine the robot's position
on the graph. The solution of this problem is the metric dimension
and a basis of the graph.
\par Now let there exist a large
number of landmarks, but the cost of computing  distance is much
for the robot. In this case, robot can determine its position on
the graph only by knowing landmarks which are adjacent to it.
Here, the problem of finding the fewest number of landmarks
needed, and where should be located, so that the adjacency and
none-adjacency to the landmarks uniquely determine the robot's
position on the graph is a different problem. The answer to this
problem is one of the motivations of introducing {\it adjacency
resolving sets} in graphs.

\begin{deff}\label{Adjacency metric dimension}
Let $G$ be a graph and $W=\{w_1,w_2,\ldots,w_k\}$ be an ordered
subset of $V(G)$. For each vertex $v\in V(G)$ the adjacency
representation of $v$ with respect to $W$ is  $k$-vector
$$r_2(v|W):=(a_G(v,w_1),a_G(v,w_2),\ldots,a_G(v,w_k)),$$ where
$$a_G(v,w_i)=\left\{
\begin{array}{ll}
0 &  {\rm if}~v=w_i, \\
1 &  {\rm if}~v\sim w_i,\\
2 &   {\rm if}~v\nsim w_i.
\end{array}\right.$$ If all distinct vertices of $G$ have distinct
adjacency representations, $W$ is called an adjacency resolving
set for $G$. The minimum cardinality of an adjacency resolving set
is called  adjacency metric dimension of $G$, denoted by
$\beta_2(G)$. An adjacency resolving set of cardinality
$\beta_2(G)$ is called an adjacency basis of $G$.
\end{deff}
\par By the definition, if $G$ is a
connected graph with diameter $2$, then $\beta(G)=\beta_2(G)$. The
converse is false; it  can be seen that
$\beta_2(C_6)=2=\beta(C_6)$ while, $diam(C_6)=3$.
\par In the following, we obtain some useful
results on the adjacency metric dimension of  graphs.
\begin{pro}\label{B(H)<B2(H)} For every connected graph $G$,   $\beta(G)\leq \beta_2(G)$.
\end{pro}
\begin{proof}{
Let $W$ be an adjacency basis of $G$. Thus, for each pair of
vertices $u,v\in V(G)$ there exist a vertex $w\in W$ such that,
$a_G(u,w)\neq a_G(v,w)$. Therefore, $d_G(u,w)\neq d_G(v,w)$ and
hence $W$ is a resolving set for $G$. }\end{proof}
\begin{pro}\label{B2(H)=B2(overlineH)}
For every graph $G$,  $\beta_2(G)=\beta_2(\overline{G})$.
\end{pro}
\begin{proof}{Let $W$ be an adjacency basis of $G$. For each pair of
vertices $u,v\in V(G)$, there exist a vertex $w\in W$ such that
$a_G(u,w)\neq a_G(v,w)$. Without loss of generality,   assume that
$a_G(u,w)< a_G(v,w)$. Thus, if $a_G(u,w)=0$, then
$a_{\overline{G}}(u,w)=0$ and $a_{\overline{G}}(v,w)>0$. Also, if
$a_G(u,w)=1$, then $a_G(v,w)=2$ and hence,
$a_{\overline{G}}(u,w)=2$ and $a_{\overline{G}}(v,w)=1$.
Therefore, $W$ is an adjacency resolving set for $\overline{G}$
and $\beta_2(\overline{G})\leq \beta_2(G)$. Since
$\overline{\overline{G}}=G$, we conclude that $\beta_2(G)\leq
\beta_2(\overline{G})$ and consequently, $\beta_2(G)=
\beta_2(\overline{G})$. }\end{proof} Let $G$ be a graph of order
$n$.  It is easy to see that, $1\leq \beta_2(G)\leq n-1$. In the
following proposition, we characterize all graphs $G$ with
$\beta_2(G)=1$ and all graphs $G$ of order $n$ and
$\beta_2(G)=n-1$.
\begin{pro}\label{characterization 1, m-1}
If $G$ is a graph of order $n$, then
\begin{description}
 \item (i) $\beta_2(G)=1$ if and only if
$G\in\{P_1,P_2,P_3,\overline{P}_2,\overline{P}_3\}$.
\item (ii) $\beta_2(G)=n-1$ if and only if $G=K_n$ or
$G=\overline{K}_n$.
\end{description}
\end{pro}
\begin{proof}{(i) It is easy to see that for
$G\in\{P_1,P_2,P_3,\overline{P}_2,\overline{P}_3\}$,
$\beta_2(G)=1$. Conversely, let $G$ be a graph with
$\beta_2(G)=1$. If $G$ is a connected graph, then by
Proposition~\ref{B(H)<B2(H)}, $\beta(G)\leq \beta_2(G)=1$. Thus,
by Theorem~\ref{B=1,B=n-1}, $G=P_n$. If $n\geq 4$, then
$\beta_2(P_n)\geq 2$. Hence, $n\leq3$. If $G$ is a disconnected
graph and $\beta_2(G)=1$, then $\overline{G}$ is a connected graph
and by Proposition~\ref{B2(H)=B2(overlineH)},
$\beta_2(\overline{G})=1$. Thus, $\overline{G}=P_n$,
$n\in\{2,3\}.$ Therefore, $G=\overline{P}_2$ or
$G=\overline{P}_3$.

\noindent (ii) By Proposition~\ref{B(H)<B2(H)}, we have
$n-1=\beta(K_n)\leq \beta_2(K_n)$. On the other hand,
$\beta_2(G)\leq n-1$. Therefore, $\beta_2(K_n)=n-1$ and by
Proposition~\ref{B2(H)=B2(overlineH)},
$\beta_2(\overline{K}_n)=\beta_2(K_n)=n-1$. Conversely, let $G$ be
a connected graph with $\beta_2(G)=n-1$. Suppose on the contrary
that $G\neq K_n$. Thus, $P_3$ is an induced subgraph of $G$. Let
$P_3=(x_1,x_2,x_3)$. Therefore, $a_G(x_2,x_1)=1$ and
$a_G(x_3,x_1)=2$. Consequently, $V(G)\backslash\{x_2,x_3\}$ is an
adjacency resolving set for $G$ of cardinality $n-2$. That is,
$\beta_2(G)\leq n-2$, which is a contradiction. Hence, $G=K_n$. If
$G$ is a disconnected graph with $\beta_2(G)=n-1$, then
$\overline{G}$ is a connected graph and by
Proposition~\ref{B2(H)=B2(overlineH)},
$\beta_2(\overline{G})=n-1$. Thus, $\overline{G}=K_n$.
}\end{proof}
\begin{lemma}\label{delta=n-1}
If $u$ is a vertex of degree $n(G)-1$ in a connected graph $G$,
then $G$ has a basis which does not include  $u$.
\end{lemma}
\begin{proof}{Let $B$ be a basis of $G$ which contains $u$.
Thus, $r(u|B\backslash\{u\})=(1,\ldots,1)$. Since $B$ is a basis
of $G$, there exist two vertices $v,w\in
V(G)\backslash(B\backslash\{u\})$ such that,
$r(v|B\backslash\{u\})=r(w|B\backslash\{u\})$ and $d_G(u,v)\neq
d_G(u,w)$. If $u\notin\{v,w\}$, then $d(u,v)=d(u,w)=1$, which is a
contradiction. Hence, $u\in\{v,w\}$, say $u=v$. Therefore,
$r(w|B\backslash\{u\})=r(u|B\backslash\{u\})=(1,1,\ldots,1)$ and
for each $x,y\in V(G)\backslash\{u,w\}$,
$r(x|B\backslash\{u\})\neq r(y|B\backslash\{u\})$. Note that,
$r(w|B)=(1,1,\ldots,1)$, because $u\sim w$. Since $B$ is a basis
of $G$, $w$ is the unique vertex of $G$ which its representation
with respect to $B$ is entirely $1$. It implies that $w$ is the
unique vertex of $G\backslash B$ with
$r(w|B\backslash\{u\})=(1,1,\ldots,1)$. Therefore, the set
$(B\backslash\{u\})\cup\{w\}$ is a basis of $G$ which does not
contain $u$.
 }\end{proof}
\begin{pro}\label{H join K_1}
For every graph $G$, $\beta(G\vee K_1)-1\leq \beta_2(G)\leq
\beta(G\vee K_1)$. Moreover, $\beta_2(G)=\beta(G\vee K_1)$ if and
only if $G$ has an adjacency basis  for which no vertex has
adjacency representation entirely $1$ with respect to it.
\end{pro}
\begin{proof}{
Let $V(G)=\{v_1,v_2,\ldots,v_n\}$ and $V(K_1)=\{u\}$. Note that,
$d_{G\vee K_1}(v_i,v_j)=a_G(v_i,v_j)$, $1\leq i,j \leq n$.  By
Lemma~\ref{delta=n-1}, $G\vee K_1$ has a basis
$B=\{b_1,b_2,\ldots,b_k\}$ such that $u\notin B$.
Therefore,$$r(v_i|B)=(d_{G\vee K_1}(v_i,b_1),d_{G\vee
K_1}(v_i,b_2),\ldots,d_{G\vee K_1}(v_i,b_k))=r_2(v_i|B)$$ for each
$v_i$, $1\leq i\leq n$. Thus, $B$ is an adjacency resolving set
for $G$ and  $\beta_2(G)\leq \beta(G\vee K_1)$.

\par Now let $W=\{w_1,w_2,\ldots,w_t\}$ be an adjacency basis of $G$.
Since $d_{G\vee K_1}(v_i,w_j)=a_G(v_i,w_j)$, $1\leq i \leq n$,
$1\leq j\leq t$, we have $r(v_i|W)=r_2(v_i|W)$, $1\leq i\leq n$.
Hence,  $W$ resolves $V(G\vee K_1)\backslash \{u\}$ and
$\beta(G\vee K_1)-1\leq \beta_2(G)$. On the other hand, $r(u|W)$
is entirely $1$. Therefore, $W$ is a resolving set for $G\vee K_1$
if and only if $r_2(v_i|W)$ is not entirely $1$ for every $v_i$,
$1\leq i\leq n$. Since $\beta_2(G)\leq \beta(G\vee K_1)$, we have
$\beta_2(G)=\beta(G\vee K_1)$ if and only if $r_2(v_i|W)$ is not
entirely $1$ for every $v_i$, $1\leq i\leq n$. }\end{proof}
\begin{pro}\label{B_2(P_m),B_2(C_m)}
If $n\geq 4$, then $\beta_2(C_n)=\beta_2(P_n)=\lfloor{{2n+2}\over
5}\rfloor$.
\end{pro}
\begin{proof}{
If $n\leq 8$, then by a simple computation, we can see that
$\beta_2(C_n)=\beta_2(P_n)=\lfloor{{2n+2}\over 5}\rfloor$. Now,
let $G\in\{P_n,C_n\}$, and  $n\geq 9$. By Theorem~\ref{B(P_n)
B(C_N)},  $\beta(G\vee K_1)=\lfloor{{2n+2}\over 5}\rfloor\geq 4$.
Hence, by Proposition~\ref{H join K_1}, we have  $\beta_2(G)\geq
3$. If $W$ is an adjacency basis of $G$, then for each vertex
$v\in V(G)$, $r_2(v|W)$ is not entirely $1$, because $v$ has at
most two neighbors. Therefore, by Proposition~\ref{H join K_1},
$\beta_2(G)=\beta_2(G\vee K_1)=\lfloor{{2n+2}\over 5}\rfloor$.
}\end{proof}

\begin{pro}\label{B_2 multipartite} If $K_{m_1,m_2,\ldots,m_t}$ is
the complete $t$-partite graph, then
$$\beta_2(K_{m_1,m_2,\ldots,m_t})=\beta(K_{m_1,m_2,\ldots,m_t})=\left\{
\begin{array}{ll}
m-r-1 &  ~{\rm if}~r\neq t, \\
m-r &  ~{\rm if}~r=t, \end{array}\right.$$ where $m_1,
m_2,\ldots,m_r$ are at least $2$, $m_{r+1}=\cdots=m_t=1$, and
$\sum_{i=1}^tm_i=m$.
\end{pro}
\begin{proof}{ Since $diam(K_{m_1,m_2,\ldots,m_t})=2$,
we have $\beta_2(K_{m_1,m_2,\ldots,m_t})=
\beta(K_{m_1,m_2,\ldots,m_t})$. Let $M_i$ be the partite set of
size $m_i$, $1\leq i\leq t$. For each $i,~1\leq i\leq r$, all
vertices of $M_i$ are none-adjacent twins. Also, all vertices of
$\cup_{i=r+1}^tM_i$ are adjacent twins. Let $x_i$ be a fixed
vertex in $M_i$, $1\leq i\leq r$. If $r=t$, then by
Observation~\ref{twins},
$\beta(K_{m_1,m_2,\ldots,m_t})\geq\sum_{i=1}^tm_i-r$. Also, the
set $\cup_{i=1}^tM_i\backslash\{x_1,x_2,\ldots,x_r\}$ is a
resolving set for $K_{m_1,m_2,\ldots,m_t}$ with cardinality
$\sum_{i=1}^tm_i-r$. Thus,
$\beta(K_{m_1,m_2,\ldots,m_t})=\sum_{i=1}^tm_i-r=m-r$. If $r\neq
t$, then $\cup_{i=r+1}^t M_i\neq\emptyset$. Let $x_{r+1}\in
\cup_{i=r+1}^t M_i$. Observation~\ref{twins} implies that
$\beta(K_{m_1,m_2,\ldots,m_t})\geq\sum_{i=1}^tm_i-r-1$. On the
other hand, the set
$\cup_{i=1}^tM_i\backslash\{x_1,x_2,\ldots,x_{r+1}\}$ is a
resolving set for $K_{m_1,m_2,\ldots,m_t}$ with cardinality
$\sum_{i=1}^tm_i-r-1=m-r-1$. }
\end{proof}

\section{Lexicographic Product of Graphs}\label{Lexicographic product}
Throughout this section, $G$ is a connected graph of order $n$,
$V(G)=\{v_1,v_2,\ldots,v_n\}$, $H$ is a graph of order $m$, and
$V(H)=\{u_1,u_2,\ldots,u_m\}$. Therefore, $G[H]$ is a connected
graph. For convenience,  we denote the vertex $(v_i,u_j)$  of
$G[H]$ by $v_{ij}$.  Note that, for each pair of vertices
$v_{ij},v_{rs}\in V(G[H])$,
$$d_{G[H]}(v_{ij},v_{rs})=\left\{
\begin{array}{ll}
d_G(v_i,v_r) &  ~{\rm if}~v_i\neq v_r, \\
1 &  ~{\rm if}~v_i=v_r~{\rm and}~u_j\sim u_s,\\
2 &   ~{\rm if}~v_i=v_r~{\rm and}~u_j\nsim u_s.
\end{array}\right.$$On the other words, $$d_{G[H]}(v_{ij},v_{rs})=\left\{
\begin{array}{ll}
d_G(v_i,v_r) &  ~{\rm if}~v_i\neq v_r, \\
a_H(u_j,u_s)&  ~{\rm otherwise.}
\end{array}\right.$$
Let $S$ be a subset of $V(G[H])$. The {\it projection} of $S$ onto
$H$ is the set $\{u_j\in V(H)\ |\ v_{ij}\in S\}$. Also, the ith
{\it row} of $G[H]$, denoted by $H_i$, is the set $\{v_{ij}\in
V(G[H])\ |\ 1\leq j\leq m\}.$
\begin{lemma}\label{S i Resolves H_i} If\- $W\subseteq V(G[H])$ is a
resolving set for $G[H]$, then $W\cap H_i$ resolves $H_i$, for
each $i,~1\leq i\leq n$. Moreover, the projection of $W\cap H_i$
onto $H$ is an adjacency resolving set for $H$, for each $i,~1\leq
i\leq n$.
\end{lemma}
\begin{proof}{
Since $W$ resolves $G[H]$, for each pair of vertices
$v_{ij},v_{iq}\in H_i$, there exist a vertex $v_{rt}\in W$ such
that, $d_{G[H]}(v_{rt},v_{ij})\neq d_{G[H]}(v_{rt},v_{iq})$. If
$r\neq i$, then $d_{G[H]}(v_{rt},v_{ij})=d_G(v_r,v_i)=
d_{G[H]}(v_{rt},v_{iq})$, which is a contradiction. Therefore,
$i=r$ and $W\cap H_i$ resolves $H_i$.
\par Now,  let $u_j,u_q\in
V(H)$. Since $W\cap H_i$ resolves $H_i$, there exist a vertex
$v_{it}\in W\cap H_i$ such that, $d_{G[H]}(v_{it},v_{ij})\neq
d_{G[H]}(v_{it},v_{iq})$. Hence,
$a_H(u_t,u_j)=d_{G[H]}(v_{it},v_{ij})\neq
d_{G[H]}(v_{it},v_{iq})=a_H(u_t,u_q)$. Consequently, the
projection of $W\cap H_i$ onto $H$ is an adjacency resolving set
for $H$. }
\end{proof}
By Lemma~\ref{S i Resolves H_i}, every  basis of $G[H]$ contains
at least $\beta_2(H)$ vertices from each copy of $H$ in $G[H]$.
Thus, the following lower bound for $\beta(G[H])$ is obtained.

\begin{equation}\label{l bound}
\beta(G[H])\geq n\beta_2(H).
\end{equation}
\begin{thm}\label{B1 B2}
Let $G$ be a connected graph of order $n$ and $H$ be an arbitrary
graph. If there exist two adjacency bases $W_1$ and $W_2$ of $H$
such that, there is no vertex with adjacency representation
entirely~$1$ with respect to $W_1$ and no vertex with adjacency
representation entirely $2$ with respect to $W_2$, then
$\beta(G[H])=\beta(G[\overline H])=n\beta_2(H)$.
\end{thm}
\begin{proof}{
By Inequality~\ref{l bound}, we have $\beta(G[H])\geq
n\beta_2(H)$. To prove the equality, it is enough to provide a
resolving set for $G[H]$ of size $n\beta_2(H)$. For this sake, let
 $$S=\{v_{ij}\in V(G[H])\ |\ v_i\in K(G),~u_j\in W_1\}\cup \{v_{ij}\in V(G[H])\
 |\ v_i\notin K(G),~u_j\in W_2\},$$
 where $K(G)$ is the set of all vertices of $G$ in equivalence
 classes of type (K). On the other word, $K(G)$ is the set of all
 vertices of $G$ which have adjacent twins. We show that $S$ is a
 resolving set for $G[H]$.
   Let $v_{rt},v_{pq}\in
V(G[H])\backslash S$ be two distinct vertices.
The following possibilities can be happened. \vspace{4mm}\\
1. $r=p$. Note that, $v_{rt}\neq v_{pq}$ implies that $t\neq q$.
Since $W_1$ and $W_2$  are adjacency resolving sets, there exist
vertices $u_j\in W_1$ and $u_l\in W_2$ such that,
$a_H(u_t,u_j)\neq a_H(u_q,u_j)$ and $a_H(u_t,u_l)\neq
a_H(u_q,u_l)$. If $v_r\in K(G)$, then $v_{rj}\in S$ and
$d_{G[H]}(v_{rt},v_{rj})=a_H(u_t,u_j)\neq
a_H(u_q,u_j)=d_{G[H]}(v_{pq},v_{rj})$. Similarly, if $v_r\notin
K(G)$, then $v_{rl}\in S$ and $d_{G[H]}(v_{rt},v_{rl})\neq
d_{G[H]}(v_{pq},v_{rl})$.\vspace{4mm}\\
2. $r\neq p$ and $v_r,v_p\in K(G)$. If $v_r$ and $v_p$ are not
twins, then there exist a vertex $v_i\in
V(G)\backslash\{v_r,v_p\}$ which is adjacent to only one of the
vertices $v_r$ and $v_p$. Hence, for each $u_j\in W_1$, we have
$v_{ij}\in S$ and $d_{G[H]}(v_{rt},v_{ij})=d_G(v_r,v_i)\neq
d_G(v_p,v_i)= d_{G[H]}(v_{pq},v_{ij})$. If $v_r$ and $v_p$ are
twins, then $v_r\sim v_p$, because $v_r,v_p\in K(G)$. Since
$r_2(u_t|W_1)$ is not entirely $1$, there exist a vertex $u_l\in
W_1$ such that, $a_H(u_t,u_l)=2$. Therefore, $v_{rl}\in S$ and
$d_{G[H]}(v_{rt},v_{rl})=a_H(u_t,u_l)=2$. On the other hand,
$d_{G[H]}(v_{pq},v_{rl})=d_G(v_p,v_r)=1$. Thus,
$d_{G[H]}(v_{rt},v_{rl})\neq d_{G[H]}(v_{pq},v_{rl})$. \vspace{4mm}\\
3. $r\neq p$, $v_r\in K(G)$, and $v_q\notin K(G)$. In this case,
$v_r$ and $v_p$ are not twins. Therefore, there exist a vertex
$v_i\in V(G)\backslash\{v_r,v_p\}$ which is adjacent to only one
of the vertices $v_r$ and $v_p$. Let $u_j$ be a vertex of $W_1\cup
W_2$, such that $v_{ij}\in S$. Hence,
$d_{G[H]}(v_{rt},v_{ij})=d_G(v_r,v_i)\neq d_G(v_p,v_i)=
d_{G[H]}(v_{pq},v_{ij})$. \vspace{4mm}\\
4. $r\neq p$ and $v_r,v_p\notin K(G)$. If $v_r$ and $v_p$ are not
twins, then there exist a vertex $v_i\in
V(G)\backslash\{v_r,v_p\}$ which is adjacent to only one of the
vertices $v_r$ and $v_p$. Thus, for each $u_j\in W_2$, we have
$v_{ij}\in S$ and $d_{G[H]}(v_{rt},v_{ij})=d_G(v_r,v_i)\neq
d_G(v_p,v_i)= d_{G[H]}(v_{pq},v_{ij})$. If $v_r$ and $v_p$ are
twins, then $v_r\nsim v_p$, because $v_r,v_p\notin K(G)$. Since
$r_2(u_t|W_2)$ is not entirely $2$, there exist a vertex $u_l\in
W_2$,  such that $a_H(u_t,u_l)=1$. Therefore, $v_{rl}\in S$ and
$d_{G[H]}(v_{rt},v_{rl})=a_H(u_t,u_l)=1$. On the other hand,
$d_{G[H]}(v_{pq},v_{rl})=d_G(v_p,v_r)=2$, since $v_r$ and $v_p$
are none-adjacent twins in the connected $G$. Hence,
$d_{G[H]}(v_{rt},v_{rl})\neq d_{G[H]}(v_{pq},v_{rl})$.
\par Thus, $r(v_{rt}|S)\neq r(v_{pq}|S)$. Therefore, $S$ is a
resolving set for $G[H]$ with cardinality $n\beta_2(H)$.
\par Clearly, in $\overline H$, for each $u\in V(\overline
H)$, $r_2(u|W_1)$ is not entirely $2$ and $r_2(u|W_2)$ is not
entirely $1$. Since $\beta_2(H)=\beta_2(\overline H)$, by
interchanging the roles of $W_1$ and $W_2$ for $\overline H$, we
conclude $\beta(G[\overline H])=n\beta_2(\overline
H)=n\beta_2(H)$. }\end{proof} In the following three theorems, we
obtain $\beta(G[H])$, when $H$ does not satisfy the assumption of
Theorem~\ref{B1 B2}.
\begin{thm}\label{thm generalG[H]}
Let $G$ be a connected graph of order $n$ and $H$ be an arbitrary
graph. If for each adjacency basis $W$ of $H$ there exist vertices
with adjacency representations entirely $1$ and entirely $2$ with
respect to $W$, then $\beta(G[H])=\beta(G[\overline
H])=n(\beta_2(H)+1)-\iota(G).$
\end{thm}
\begin{proof}{ Let $B$ be a basis of $G[H]$ and $B_i$ be the
projection of $B\cap H_i$ onto $H$, for each $i$, $1\leq i\leq n$.
By Lemma~\ref{S i Resolves H_i}, $B_i$'s are adjacency resolving
sets for $H$. Therefore, $|B\cap H_i|=|B_i|\geq \beta_2(H)$ for
each $i$, $1\leq i\leq n$.

Let $I=\{i\ |\  |B_i|=\beta_2(H)\}$. We claim that
$|I|\leq\iota(G)$, otherwise by the pigeonhole principle, there
exist a pair of twin vertices $v_r,v_p\in V(G)$ such that,
$|B_r|=|B_p|=\beta_2(H)$. Since  $B_r$ and $B_p$ are adjacency
bases of $H$, by the assumption there are vertices $u_t$ and
$u_q$ with adjacency representations entirely $1$ with respect to
$B_r$ and $B_p$, respectively. Also, there are vertices $u_t'$ and
$u_q'$ with adjacency representations entirely $2$ with respect to
$B_r$ and $B_p$, respectively.
 Hence, for each $u\in B_r$ and $u^\prime\in B_p$, we have
$u_{t}\sim u,~u_{t^\prime}\nsim u,~u_{q}\sim u^\prime$, and
$u_{q^\prime}\nsim u^\prime$. If $v_r\sim v_p$, then for each
$v_{ij}\in B$ one of the following
cases can be happened.\vspace{4mm}\\
1. $i\notin\{r,p\}$. Since $v_r$ and $v_p$ are twins, we have
$d_G(v_r,v_i)=d_G(v_p,v_i)$. On the other hand,
$d_{G[H]}(v_{rt},v_{ij})=d_G(v_r,v_i)$ and
$d_{G[H]}(v_{pq},v_{ij})=d_G(v_p,v_i)$. Thus,
$d_{G[H]}(v_{rt},v_{ij})=d_{G[H]}(v_{pq},v_{ij})$.\vspace{4mm}\\
2. $i=p\neq r$. In this case,
$d_{G[H]}(v_{pq},v_{ij})=a_H(u_{q},u_j)$ and
$d_{G[H]}(v_{rt},v_{ij})=d_G(v_r,v_i)$. Since $v_i=v_p\sim v_r$,
we have $d_G(v_r,v_i)=1$. On the other hand $u_j\in B_p$ and
hence, $a_H(u_{q},u_j)=1$. Therefore,
$d_{G[H]}(v_{rt},v_{ij})=d_{G[H]}(v_{pq},v_{ij})$.\vspace{4mm}\\
3. $i=r\neq p$. Similar to previous  case,
$d_{G[H]}(v_{rt},v_{ij})=a_H(u_{t},u_j)=1$ and
$d_{G[H]}(v_{pq},v_{ij})=d_G(v_p,v_i)=1$. Consequently,
$d_{G[H]}(v_{rt},v_{ij})=d_{G[H]}(v_{pq},v_{ij})$.\vspace{4mm}\\
4. $i=p=r$. In this case, $d_{G[H]}(v_{pq},v_{ij})=a_H(u_{q},u_j)$
and $d_{G[H]}(v_{rt},v_{ij})=a_H(u_{t},u_j)$. Since, $u_j\in
B_p=B_r$, we have $a_H(u_{q},u_j)=1=a_H(u_{t},u_j)$. Thus,
$d_{G[H]}(v_{rt},v_{ij})=d_{G[H]}(v_{pq},v_{ij})$.
\vspace{4mm}\\
Hence, $v_r\sim v_p$ implies that $r(v_{rt}|B)=r(v_{pq}|B)$, which
is a contradiction. Therefore, $v_r\nsim v_p$. Since $G$ is a
connected graph, none-adjacent twin vertices $v_r$ and $v_p$ have
at least one common neighbor and thus, $d_G(v_r,v_p)=2$.
Consequently, by a same method as the case $v_r\sim v_p$, we can
see that $r(v_{rt^\prime}|B)=r(v_{pq^\prime}|B)$, which
contradicts the assumption that $B$ is a basis of $G[H]$. Hence
$|I|\leq \iota(G)$. On the other hand, every basis of $G[H]$ has
at least $\beta_2(H)+1$ vertices in $H_i$, where $i\notin I$.
 Therefore,
\begin{eqnarray*}
\beta(G[H])=|B|=|\cup_{i=1}^n(B\cap H_i)|&\geq &|I|\beta_2(H)+(n-|I|)(\beta_2(H)+1)\\
&=&n\beta_2(H)+n-|I|\\
&\geq & n(\beta_2(H)+1)-\iota(G).
\end{eqnarray*}
Now let $W$ be an adjacency basis of $H$. By assumption, there
exist vertices  $u_1,u_2\in V(H)\backslash W$ such that, $u_1$ is
adjacent to all vertices of $W$ and $u_2$ is not adjacent to any
vertex of $W$. Also, let $K(G)$ be the set of all classes  of
type (K), and $N(G)$ be the set of all classes of $G$ of type (N)
in $G$. Choose fixed vertex $v$ from $v^*$ for each $v^*\in
N(G)\cup K(G)$. We claim that the set
$$S=\{v_{ij}\in V(G[H])\,|\,u_j\in W\}\cup\{v_{t1}\,|\,v_t\in
\cup_{v^*\in K(G)}(v^*\backslash\{v\})\} \cup\{v_{t2}\,|\,v_t\in
\cup_{v^*\in N(G)}(v^*\backslash\{v\})\}$$ is a resolving set for
$G[H]$. Let $v_{rt},v_{pq}\in V(G[H])\backslash S$. Hence, one
of the following cases can be happened.\vspace{4mm}\\
1. $r=p$. Since $W$ is an adjacency basis of $H$, there exist a
vertex $u_j\in W$,  such that $a_H(u_q,u_j)\neq a_H(u_t,u_j)$.
Therefore, $d_{G[H]}(v_{pq},v_{rj})=a_H(u_q,u_j)\neq
a_H(u_t,u_j)=d_{G[H]}(v_{rt},v_{rj})$. Consequently,
$r(v_{rt}|S)\neq r(v_{pq}|S)$. \vspace{4mm}\\
2. $r\neq p$ and $v_r,v_p$ are not twins. Hence, there exist a
vertex $v_i\in V(G)$ which is adjacent to only one of the vertices
$v_r$ and $v_p$. Thus, for each vertex $u_j\in W$,
$d_{G[H]}(v_{rt},v_{ij})=d_G(v_r,v_i)\neq
d_G(v_p,v_i)=d_{G[H]}(v_{pq},v_{ij})$.
This yields, $r(v_{rt}|S)\neq r(v_{pq}|S)$. \vspace{4mm}\\
3. $v_r$ and $v_p$ are adjacent twins. Therefore, at least one of
the vertices $v_{r1}$ and $v_{p1}$, say $v_{r1}$ belongs to $S$.
Since $v_{rt}\notin S$, we have $t\neq 1$. Hence, there exists a
vertex $u_j\in S$, such that $a_H(u_t,u_j)=2$, otherwise $t=1$.
Consequently, $d_{G[H]}(v_{rt},v_{rj})=a_H(u_t,u_j)=2$. On the
other hand, $d_{G[H]}(v_{pq},v_{rj})=d_G(v_p,v_r)=1$, because
$v_r\sim v_p$. This gives, $r(v_{rt}|S)\neq r(v_{pq}|S)$. \vspace{4mm}\\
4. $v_r$ and $v_p$ are none-adjacent twins. In this case, at least
one of the vertices $v_{r2}$ and $v_{p2}$, say $v_{r2}$ belongs to
$S$. Hence, $t\neq 2$ and there exists a vertex $u_j\in W$, such
that $a_H(u_t,u_j)=1$, otherwise $t=2$. Therefore,
$d_{G[H]}(v_{rt},v_{rj})=a_H(u_t,u_j)=1\neq2=d_G(v_p,v_r)=d_{G[H]}(v_{pq},v_{rj})$.
 Thus, $r(v_{rt}|S)\neq r(v_{pq}|S)$. \vspace{4mm}\\
Consequently, $S$ is a resolving set for $G[H]$ with cardinality
$$|S|=n\beta_2(H)+a(G)-\iota_{_K}(G)+b(G)-\iota_{_N}(G)=n(\beta_2(H)+1)-\iota(G). $$ Since all adjacency bases of $H$ and
$\overline H$ are the same, $\overline H$ satisfies the condition
of the theorem. Hence, $\beta(G[\overline
H])=n(\beta_2(H)+1)-\iota(G)$ and the proof is completed.
}\end{proof}
\begin{thm}\label{nB2+a(G) K(G)}
Let $G$ be a connected graph of order $n$ and $H$ be an arbitrary
graph. If $H$ has the following properties
\begin{description}\item (i) for each adjacency basis of $H$
there exist a vertex with adjacency representation entirely $1$,
\item (ii) there exist an adjacency basis $W$ of $H$ such that
there is no vertex with adjacency representation entirely $2$ with
respect to $W$,
\end{description} then
$\beta(G[H])=n\beta_2(H)+a(G)-\iota_{_K}(G).$
\end{thm}
\begin{proof}{
Let $B$ be a basis of $G[H]$ and $B_i$ be the projection of $B\cap
H_i$ onto $H$, for each $i$, $1\leq i\leq n$. By Lemma~\ref{S i
Resolves H_i}, $B_i$'s are adjacency resolving sets for $H$.
Therefore, $|B\cap H_i|=|B_i|\geq \beta_2(H)$ for each $i$, $1\leq
i\leq n$.

Let $I=\{i\ |\  |B_i|=\beta_2(H)\}$. We claim that $|I|\leq
n-a(G)+\iota_{_K}(G)$, otherwise by the pigeonhole principle,
there exist a pair of adjacent twin vertices $v_r,v_p\in V(G)$,
such that $|B_r|=|B_p|=\beta_2(H)$. Since  $B_r$ and $B_p$ are
adjacency bases of $H$, by assumption $(i)$    there exist
vertices $u_{t},u_q\in V(H)$ with adjacency representation
entirely $1$ with respect to $B_r$ and $B_p$, respectively.
Hence, for each $u\in B_r$ and each $u^\prime\in B_p$, we have
$u_{t}\sim u$, and $u_q\sim u'$. Since $v_r\sim v_p$, for each
$v_{ij}\in B$ one of the following
cases can be happened.\vspace{4mm}\\
1. $i\notin\{r,p\}$. Since $v_r$ and $v_p$ are twins, we have
$d_G(v_r,v_i)=d_G(v_p,v_i)$. On the other hand,
$d_{G[H]}(v_{rt},v_{ij})=d_G(v_r,v_i)$ and
$d_{G[H]}(v_{pq},v_{ij})=d_G(v_p,v_i)$. Thus,
$d_{G[H]}(v_{rt},v_{ij})=d_{G[H]}(v_{pq},v_{ij})$.\vspace{4mm}\\
2. $i=p\neq r$. In this case,
$d_{G[H]}(v_{pq},v_{ij})=a_H(u_{q},u_j)$ and
$d_{G[H]}(v_{rt},v_{ij})=d_G(v_r,v_i)$. Since $v_i=v_p\sim v_r$,
we have $d_G(v_r,v_i)=1$. On the other hand $u_j\in B_p$ and
hence, $a_H(u_{q},u_j)=1$. Therefore,
$d_{G[H]}(v_{rt},v_{ij})=d_{G[H]}(v_{pq},v_{ij})$.\vspace{4mm}\\
3. $i=r\neq p$. Similar to previous  case,
$d_{G[H]}(v_{rt},v_{ij})=a_H(u_{t},u_j)=1$ and
$d_{G[H]}(v_{pq},v_{ij})=d_G(v_p,v_i)=1$. Consequently,
$d_{G[H]}(v_{rt},v_{ij})=d_{G[H]}(v_{pq},v_{ij})$.\vspace{4mm}\\
4. $i=p=r$. In this case, $d_{G[H]}(v_{pq},v_{ij})=a_H(u_{q},u_j)$
and $d_{G[H]}(v_{rt},v_{ij})=a_H(u_{t},u_j)$. Since, $u_j\in
B_p=B_r$, we have $a_H(u_{q},u_j)=1=a_H(u_{t},u_j)$. Thus,
$d_{G[H]}(v_{rt},v_{ij})=d_{G[H]}(v_{pq},v_{ij})$.
\vspace{4mm}\\
Hence, $r(v_{rt}|B)=r(v_{pq}|B)$, which is a contradiction.
Therefore, $|I|\leq n-a(G)+\iota_{_K}(G)$.   On the other hand,
every basis of $G[H]$ has at least $\beta_2(H)+1$ vertices in
$H_i$, where $i\notin I$. Thus,
\begin{eqnarray*}
\beta(G[H])=|B|&\geq &|I|\beta_2(H)+(n-|I|)(\beta_2(H)+1)\\
&=&n\beta_2(H)+n-|I|\\
&\geq & n\beta_2(H)+a(G)-\iota_{_K}(G).
\end{eqnarray*}
Now let $K(G)$ be the set of all classes  of type (K) in $G$ and
$v\in v^*$ be a fixed vertex for each class $v^*$ of type (K).
Also, let $u_1\in V(H)\backslash W$, such that $r_2(u_1|W)$ is
entirely $1$. Consider $$S=\{v_{ij}\in V(G[H])\,|\,u_j\in
W\}\cup\{v_{t1}\,|\,v_t\in \cup_{v^*\in
K(G)}(v^*\backslash\{v\})\}$$ and let $v_{rt},v_{pq}\in
V(G[H])\backslash S$. If $v_r$ and $v_p$ are not none-adjacent
twins, then similar to the proof of Theorem~\ref{thm generalG[H]},
we have $r(v_{rt}|S)\neq r(v_{pq}|S)$. Now, let $v_r$ and $v_p$ be
none-adjacent twin vertices of $G$. By assumption, there exists a
vertex $u_j\in W$,  such that $a_H(u_t,u_j)=1$. Therefore,
$d_{G[H]}(v_{rt},v_{rj})=a_H(u_t,u_j)=1$. On the other hand,
$d_{G[H]}(v_{pq},v_{rj})=d_G(v_p,v_r)=2$, since $v_r$ and $v_p$
are none-adjacent twins in the connected graph $G$. Hence,
$r(v_{rt}|S)\neq r(v_{pq}|S)$. This implies that $S$ is a
resolving set for $G[H]$ with cardinality
$n\beta_2(H)+a(G)-\iota_{_K}(G)$.
 }\end{proof}
 By a similar proof, we have the following theorem.
\begin{thm}\label{nB2+b(G) N(G)}
Let $G$ be a connected graph of order $n$ and $H$ be an arbitrary
graph. If $H$ has the following properties
\begin{description}\item (i) for each adjacency basis of $H$
there exist a vertex with adjacency representation entirely $2$,
\item (ii) there exist an adjacency basis $W$ of $H$ such that
there is no vertex with adjacency representation entirely~$1$ with
respect to $W$,
\end{description} then
$\beta(G[H])=n\beta_2(H)+b(G)-\iota_{_N}(G).$
\end{thm}

\begin{cor}\label{no twin} If $G$  has no  pair  of twin
vertices, then $\beta (G[H])=n\beta_2(H)$.
\end{cor}
\begin{proof}{ The
adjacency bases of $H$ satisfy one of the conditions of Theorems
~\ref{B1 B2},~\ref{thm generalG[H]},~\ref{nB2+a(G) K(G)},
and~\ref{nB2+b(G) N(G)} . Now, if $G$ does not have any pair of
twin vertices, then $\iota(G)=n,~\iota_{_K}(G)=a(G)=0$, and
$\iota_{_N}(G)=b(G)=0$. Therefore, $\beta
(G[H])=n\beta_2(H)$.}\end{proof} By Theorems~\ref{B1 B2},~\ref{thm
generalG[H]},~\ref{nB2+a(G) K(G)}, and~\ref{nB2+b(G) N(G)} the
exact value of $\beta(G[H])$ of many graphs $G$ and $H$ can be
determined. In the following two corollaries, $\beta(G[H])$ for
some of the well known graphs are obtained.
\begin{cor}\label{G=P_n or C_n} Let $G=P_n$, $n\geq4$ or  $G=C_n$,
$n\geq5$. Then, $G$ does not have any pair of twin vertices. Thus
by Corollary~\ref{no twin}, $\beta(G[H])=n\beta_2(H)$, for each
graph $H$. In particular, by
Propositions~\ref{B2(H)=B2(overlineH)} and
\ref{B_2(P_m),B_2(C_m)},
$\beta_2(P_m)=\beta_2(C_m)=\beta_2(\overline
P_m)=\beta_2(\overline C_m)=\lfloor{2m+2\over5}\rfloor$.
Therefore, $\beta(G[P_m])=\beta(G[C_m])=\beta(G[\overline
P_m])=\beta(G[\overline C_m])=n\lfloor{2m+2\over5}\rfloor$. Also,
by Propositions~\ref{B2(H)=B2(overlineH)} and \ref{B_2
multipartite}, we have
$$\beta(G[\overline K_{m_1,m_2,\ldots,m_t}])=\beta(G[K_{m_1,m_2,\ldots,m_t}])=\left\{
\begin{array}{ll}
n(m-r-1) &  ~{\rm if}~r\neq t, \\
n(m-r) &  ~{\rm if}~r=t, \end{array}\right.$$ where $m_1,
m_2,\ldots,m_r$ are at least $2$, $m_{r+1}=\cdots=m_t=1$, and
$\sum_{i=1}^tm_i=m$.
\end{cor}
\begin{cor}\label{H=K_m1,m2,..,mt} Let $H=K_{m_1,m_2,\ldots,m_t}$, where $m_1,
m_2,\ldots,m_r$ are at least $2$, $m_{r+1}=\cdots=m_t=1$, and
$\sum_{i=1}^tm_i=m$. Thus, for each adjacency basis  of $H$ there
is no vertex of $H$ with adjacency representation entirely $2$.
\par If $r=t$, then  for each adjacency basis  of $H$ there is no
vertex of $H$ with adjacency representation entirely $1$.
Therefore, by Theorem~\ref{B1 B2}, $\beta(G[H])=n\beta_2(H)$ for
each connected graph $G$ of order $n$.  If $r\neq t$, then for
each adjacency basis  of $H$, there exist a vertex with adjacency
representation entirely~$1$. Thus, by Theorem~\ref{nB2+a(G)
K(G)}, $\beta(G[H])=n\beta_2(H)+a(G)-\iota_{_K}(G)$ for each
connected graph $G$ of order $n$.
\par In particular, if $G=K_n$,
then all vertices of $K_n$ are adjacent twins. Thus, $a(K_n)=n$
and $\iota_{_K}(K_n)=1$, hence, $\beta(K_n[H])=n\beta_2(H)+n-1$.
Therefore, by Proposition~\ref{B_2 multipartite},
$$\beta(K_n[H])=\left\{
\begin{array}{ll}
n(m-r)-1 &  ~{\rm if}~r\neq t, \\
n(m-r) &  ~{\rm if}~r=t. \end{array}\right.$$
\end{cor}

\end{document}